\documentclass[11pt]{article}

\usepackage{amsmath,amsthm,verbatim,amssymb,amsfonts,amscd, graphicx}
\usepackage{graphics}
\usepackage{tensor}
\usepackage{enumitem}
\usepackage[utf8]{inputenc}
\usepackage{graphicx}
\usepackage{color}
\usepackage{authblk}
\usepackage{mathtools}

\theoremstyle{plain}
\newtheorem{theorem}{Theorem}[section]
\newtheorem{corollary}[theorem]{Corollary}
\newtheorem{lemma}[theorem]{Lemma}

\newtheorem{proposition}[theorem]{Proposition}

\theoremstyle{definition}
\newtheorem{definition}[theorem]{Definition}
\theoremstyle{definition}
\newtheorem{example}[theorem]{Example}

\newtheorem{remark}[theorem]{Remark}

\def\R{\mathbb{R}}

\def\setminus{/}

\def\be{\begin{equation}}
\def\ee{\end{equation}}
\def\m{\mathfrak{m}}
\def\p{\partial}

\newcounter{mnotecount}[section]

\begin{document}

\title{A Positive Mass Theorem for Manifolds with Boundary}

\author[$\dag$]{Sven Hirsch}
\author[*]{Pengzi Miao}

\affil[$\dag$]{Department of Mathematics, Duke University, Durham, NC}
\affil[*]{Department of Mathematics, University of Miami, Coral Gables, FL}

\date{}

\maketitle

\begin{abstract}
\noindent
We derive a positive mass theorem for asymptotically flat manifolds with  boundary whose mean curvature satisfies 
a sharp estimate involving the conformal Green's function.
The theorem also holds if the conformal Green's function is replaced by the standard Green's function 
for the Laplacian operator.
As an application, we obtain an inequality relating the mass and harmonic functions 
that generalizes H.~Bray's mass-capacity inequality
in his proof of the Riemannian Penrose conjecture. 
\end{abstract}

\section{Introduction}

One of the fundamental 
results in mathematical relativity is the Riemannian positive mass theorem:

\begin{theorem}[Theorem 5.3 in \cite{PMT}] \label{thm-pmt-recent}
Suppose $(M^n,g)$ is an $n$-dimensional, $ n \ge 3$, 
complete,  asymptotically flat manifold with non-negative scalar curvature. Then the ADM mass  of $(M, g)$ is non-negative and is zero  if and only if $(M^n,g)$ is isometric to the Euclidean space $(\R^n,\delta)$.
\end{theorem}

This  result was first proven in low dimensions by R.~Schoen and S.-T.~Yau  \cite{old, old7} using minimal surface methods  and later by E.~Witten  \cite{witten} for spin manifolds exploiting the Bochner-Weitzenb\"{o}ck formula for the Dirac operator (see also \cite{witten2}, \cite{witten3}). 
Recently, R.~Schoen and S.-T.~Yau  \cite{PMT} extended their arguments to arbitrary dimensions
(see  \cite{Lo3, Lo4} by J. Lohkamp for a different approach).

The minimal surface method applies to the case in which the manifold has nonempty  boundary with non-positive mean curvature, i.e. the mean curvature vector vanishes or points toward the infinity (see \cite{old}).
If the  boundary has positive mean curvature, the positivity of the mass is a  more subtle question as the boundary no longer acts as  a barrier for minimal surfaces.
For instance, 
consider a  manifold  $M$ obtained by cutting out a rotationally symmetric ball in a negative-mass 
Schwarzschild manifold.  
In this case, $M$ has negative mass and a zero area singularity (see \cite{bray-jauregui}) 
is shielded by the boundary due to the  relatively large positive mean curvature of $\partial M$. 

The spinor method also adapts to the setting of manifolds with boundary.
If  the boundary mean curvature is non-positive, positivity of the mass was shown in \cite{blackhole pmt} 
(see also \cite{blackhole pmt2}). 
Refining the spinor analysis, M.~Herzlich  \cite{penrose like}  relaxed  the requirement of  non-positive mean curvature  and showed that a condition 
 $H\le 4\sqrt{\frac\pi{|\partial M|}}$ in $3$-dimension implies the positivity of the mass. 
For higher dimensional spin manifolds, M.~Herzlich proved a similar result in \cite{Herzlich3}.

In this work, we consider asymptotically flat  manifolds of  dimensions $n\ge3$, with boundary
allowed to have positive mean curvature. We show that if the mean curvature satisfies 
an upper estimate involving the conformal Green's function, then the manifold has non-negative mass. 
More precisely, we have

\begin{theorem} \label{thm-intro}
Let  $(M^n,g)$ be an $n$-dimensional, asymptotically flat  manifold with non-negative scalar curvature $R$, with  boundary $\Sigma$. Let $H$ be the  mean curvature of $ \Sigma$ with respect to 
the $\infty$-pointing unit normal  $\nu$. 
Let $u$ be the conformal Green's function given by
$$
\left\{ 
\begin{array}{rcc}
\Delta u - \frac{n-2}{4(n-1)}Ru & = & 0  \ \ \mathrm{in} \ M \\
u & \to & 0 \ \ \mathrm{at} \ \infty \\ 
u & = & 1 \ \ \mathrm{at} \ \Sigma .
\end{array}
\right.
$$
Then the estimate 
\be \label{eq-main-bdry-cond}
H(x)\le-\frac{n-1}{n-2}\nabla_\nu u(x) \quad\text{for all $x\in \Sigma$}
\ee
implies positivity of the ADM mass  of $(M^n,g)$;
 moreover,  $(M, g)$ has zero  mass  if and only if  $(M^n,g)$  is isometric to $ (\R^n, \delta) $ minus a round ball.
\end{theorem}

By the assumption $ R \ge 0 $ and the maximum principle,  we have an immediate corollary:

\begin{corollary} \label{cor-intro}
The same result as above holds true if one replaces $u$ with the standard Green's function
$$
\left\{ 
\begin{array}{rcc}
\Delta v & = & 0  \ \ \mathrm{in} \ M \\
v & \to & 0 \ \ \mathrm{at} \ \infty \\ 
v & = & 1 \ \ \mathrm{at} \ \Sigma .
\end{array}
\right.
$$
and require
\be \label{condition2}
H(x)\le-\frac{n-1}{n-2}\nabla_\nu v(x) \quad\text{for all $x\in \Sigma$}.
\ee
\end{corollary}

In \cite{penrose2}, H.~Bray proved the following theorem which plays a key role in his proof of the Riemannian 
Penrose inequality.

\begin{theorem}[Bray \cite{penrose2}] \label{thm-bray}
Let $(M^3, g)$ be a complete, asymptotically flat $3$-manifold with nonnegative scalar curvature, 
 with boundary $\Sigma$ which has zero mean curvature.
Let $\varphi (x) $ be a function on $(M^3, g)$ which satisfies
$$
\left\{ 
\begin{array}{rcc}
\Delta \varphi & = & 0  \ \ \mathrm{in} \ M \\
\varphi & \to & 1 \ \ \mathrm{at} \ \infty \\ 
\varphi & = & 0 \ \ \mathrm{at} \ \Sigma .
\end{array}
\right.
$$
Then 
\be \label{eq-bray}
\m  \ge \mathcal C ,
\ee
 where $ \m$ is the ADM mass of $(M^3, g)$ and $ \mathcal C > 0 $ is the constant in the asymptotic expansion
$$
\varphi (x) = 1 - \frac{\mathcal C}{ |x| } + o(|x|^{-1} ) , \ \mathrm{as} \ x \to \infty. 
$$
Moreover, equality in \eqref{eq-bray} holds if and only if $(M^3, g)$ is isometric to 
a spatial Schwarzschild manifold outside its horizon. 
\end{theorem}

Using Corollary \ref{cor-intro}, we obtain the following generalization of Bray's theorem. 

\begin{theorem} \label{thm-sm}
Let $(M^n,g)$ be an $n$-dimensional, asymptotically flat  manifold, with nonnegative scalar curvature, with  boundary $\Sigma$. 
Let $H$ be the  mean curvature of $ \Sigma$ with respect to  the $\infty$-pointing unit normal  $\nu$.
Let $\phi (x) $ be a function on $(M, g)$ which satisfies
$$
\left\{ 
\begin{array}{rcc}
\Delta  \phi & = & 0  \ \ \mathrm{in} \ M \\
\phi & \to & 1 \ \ \mathrm{at} \ \infty \\ 
\phi & = & c \ \ \mathrm{at} \ \Sigma ,
\end{array}
\right.
$$
where $ c > -1 $ is a constant and $ c \neq 1$.
If 
\be \label{eq-bd-c-f}
\frac{2c}{1 - c^2} \frac{\p \phi }{\p \nu} \ge \frac{n-2}{n-1} H,
\ee
then 
\be \label{eq-ms}
 \m  \ge C , 
\ee
where $\m $ is the ADM mass of $(M^n, g)$ and  $ C$ is the constant in
$$
\phi = 1 - \frac{C}{|x|^{n-2} } + o(|x|^{2-n} ), \ \mathrm{as} \ x \to \infty. 
$$
Moreover, equality in \eqref{eq-ms} holds if and only if $(M^n, g)$ is isometric to 
the exterior region outside a rotationally symmetric sphere in an  
$n$-dimensional spatial Schwarzschild manifold, which is  
$$ \left( \R^n \setminus \left\{ | x | < r_0 \right\} , \left( 1 + \frac{\m}{2 |x|^{n-2} } \right)^\frac{4}{n-2} \delta_{ij}  \right) 
$$
for some constants $ r_0 > 0 $ 
\end{theorem}

\begin{remark}
The mass $\m$ of the Schwarzschild manifold in the  rigidity statement  of Theorem \ref{thm-sm}
can be arbitrary, in particular $\m $ can be negative. Moreover, if $ \m > 0 $, $r_0 $ can be arbitrary, meaning that
$\Sigma$ can be either outside the Schwarzschild horizon or inside the Schwarzschild horizon.
\end{remark}

\begin{remark}
If $ H \le 0 $, one can take $ c = 0 $. In this case, Theorem \ref{thm-sm} reduces to Theorem \ref{thm-bray}.
\end{remark}

\begin{remark}
An equivalent formulation of Theorem \ref{thm-sm} shows that Corollary \ref{cor-intro}  corresponds to 
a version of  Theorem \ref{thm-sm} if $c=1$.  See Theorem \ref{thm-sm-2} in Section \ref{sec-app} for 
details. 
\end{remark}

\begin{remark}
In application, given a harmonic function $ \phi $ that goes to $1$ at infinity on 
an asymptotically flat manifold with nonnegative scalar curvature, one can consider
the level sets $\Sigma_c = \phi^{-1} (c) $. As long as  condition \eqref{eq-bd-c-f}
holds at some $\Sigma_c$,  one will have an estimate of the mass  in terms of $\phi$.
\end{remark}

We now outline the proof of Theorem \ref{thm-intro}. 
The idea is simply to show the manifold $(M, g)$ in Theorem \ref{thm-intro}
admits a legitimate compact fill-in $(\Omega, \tilde g)$ so that, if 
$(\Omega, \tilde g)$ is glued to $(M, g)$, the resulting manifold satisfies 
the assumptions of the positive mass theorem.
To produce such a fill-in,  one conformally deforms  $(M, g)$ using the given function 
$u$ (or $v$).  The idea of conformally deforming an asymptotically flat manifold 
with boundary to produce a compact piece appeared in the pioneer work of 
G.~Bunting and A.~Masood-ul-Alam \cite{blackhole} and was used by H. Bray \cite{penrose2}
in his proof of the Riemannian Penrose inequality. 
A recent application of this idea to obtain capacity estimates was given by 
C. Mantoulidis, L.-F. Tam and the second author  \cite{capacity fill ins}.

Once Theorem \ref{thm-intro} is proved, Theorem \ref{thm-sm} follows  by considering 
a  conformally deformed metric $ \tilde g =  \left( \frac{1 + \phi }{2} \right)^{\frac{4}{n-2}}  g$ on $M$
and applying  Corollary \ref{cor-intro} to $(M^n, \tilde g)$.

\vspace{.2cm}

The rest of this paper is organized as follows.
We provide some background material in Section 2.
In Section 3, we prove Theorem \ref{thm-intro} for harmonically flat metrics. 
The general case of Theorem \ref{thm-intro} is proved in Section 4.  
In Section \ref{sec-app}, we use Corollary \ref{cor-intro} to prove Theorem \ref{thm-sm}.
In Section \ref{sec-end}, we give some examples and  discussions.

\vspace{.2cm}

{\bf Acknowledgements} 
This work was initiated when the first author was visiting the University of Miami  (UM) 
in October 2018.
He is grateful for the hospitality of the department of mathematics at UM.
He also wants to thank H.~Bray for many encouraging discussions 
and interest in this work. 
Research of the second author is partially supported by  Simons Foundation Collaboration Grant for Mathematicians \#585168.

\section{Prerequisites}
We start by recalling  the definition of a manifold being asymptotically flat (see \cite{PMT, penrose3}
for instance).
Here we only consider manifolds with one end, and the general case can be treated in the same fashion.
\begin{definition}
Let $n\ge3$. A Riemannian manifold $(M^n, g)$  is asymptotically flat if there is a compact set $K\subset M$ such that $M\setminus K$ is diffeomorphic to $ \R^n$ minus a ball and in this coordinate chart, $g$ satisfies
$$
g= \delta +\mathcal O(|x|^{-\tau}), \ 
\partial g=\mathcal O(|x|^{-\tau-1}),\ 
\partial^2g=\mathcal O(|x|^{-\tau-2}),\\
$$
and $ R = \mathcal O(|x|^{-q})$, 
where $\tau>\frac{n-2}2$ and $q>n$. Here $ R$ is the scalar curvature of $g$.
\end{definition}

On an asymptotically flat $(M^n, g)$, the ADM mass  \cite{ADM} $m$ is given by 
$$
\m = \frac{1}{2 (n-1) \omega_{n-1} } \lim_{r \to \infty} \int_{|x| = r}  \left( g_{ij,j} - g _{jj, i} \right) \nu^i,
$$
where $ \omega_{n-1}$ is the volume of the unit sphere in $ \R^n$, and 
the unit normal $\nu$ and volume integral are with respect to the Euclidean metric. 
The fact that $m$ is a geometric invariant of $(M, g)$ was shown by 
Bartnik \cite{witten3} and by Chru\'{s}ciel \cite{Chrusciel} independently.

We next address some regularity questions regarding the positive mass theorem. 
Originally the result was proved for smooth manifolds, but since then there have been also several singular cases, proven in \cite{corners},  \cite{shi tam2},  \cite{lee}, \cite{corners2}, \cite{lee-lefloch},
 \cite{shi tam},  and \cite{singularPMT}.
Combined with  Theorem \ref{thm-pmt-recent}, 
results and proofs in \cite{corners},  \cite{corners2} and  \cite{shi tam}
in particular give  the following theorems:

\begin{theorem}[Theorem 1 in \cite{corners}, Theorem 2 in \cite{corners2}]\label{thm-corner}
Let $(M^n,g)$ be an asymptotically flat manifold.
Suppose $ \Sigma^{n-1} \subset M^n $ is a closed  hypersurface which divides $M$ into an exterior part $M_1$ 
and an interior part $M_2$. Suppose the metric $g$ is smooth up to $ \Sigma$ on both sides of $\Sigma$
and has non-negative scalar curvature away from $\Sigma$.
Let $H_1$, $H_2$ be the mean curvature of $ \Sigma$ with resect to the unit normal pointing to infinity 
in $(M_1, g)$ and $(M_2, g)$, respectively. 
If
$$
H_1\le H_2,
$$
then  the ADM mass of $(M^n , g)$ is non-negative and is zero  if and only if $(M_1, g)$ is isometric 
to $(\R^n\setminus \Omega,\delta)$ for a bounded domain  $\Omega \subset \R^n $ with smooth boundary 
and $(M_2, g)$ is isometric to $(\Omega, \delta)$.
\end{theorem}

\begin{theorem}[Theorem 7.2 in \cite{shi tam}]\label{thm-shitam}
Let $(M^n,g)$ be an asymptotically flat manifold such that  $g\in W^{1,p}_{\operatorname{loc}}$ for some $p>n$ and is  smooth with non-negative scalar curvature away from a point $q$.
Then  $m\ge0$ and $ m=0 $ only if   $M^n$ is diffeomorphic to $\R^n$ and $g$ is flat away from $q$.
\end{theorem}

\section{The Harmonically Flat Case} \label{sec-H-flat}

As in \cite{penrose2}, we first prove Theorem \ref{thm-intro} 
for the case 
$g=\xi^{\frac 4{n-2}} \delta $ near $\infty$,  where $\xi$ is a Euclidean harmonic function. 
 This property  is often  known as $(M^n,g)$ being harmonically  flat (see \cite{penrose2}).
 
\begin{proof}
Let $(\tilde M^n,\tilde g)$ be the one point compactification of $(M^n,u^{\frac4{n-2}}g)$ where $u$ is the conformal Green's function from Theorem 1.2.
Note that in particular $\tilde g|_\Sigma=g|_\Sigma$ and by the maximum principle $u>0$.
Also, we are able  to extend $\tilde g$ smoothly to the point at $\infty$ as H.~Bray's proof of Theorem 1.4, also see the detailed exposition in section 2 of \cite{capacity fill ins}.
This argument is based on the removable singularity theorem for harmonic functions and thus crucially requires the manifold to be harmonically flat.
Note that in this case $u$ is harmonic near $\infty$.
In the general setting this conformal blow down may produce a singular metric which requires to be smoothed which is performed in section 4 following \cite{shi tam} and \cite{capacity fill ins}. 

The conformally transformed manifold $\tilde M$ has zero scalar curvature $\tilde R$ due to the well known formula
\begin{align}\label{3.1}
\tilde R=u^{-\frac 4{n-2}}R-\frac{4(n-1)}{n-2}u^{-\frac{4}{n-2}-1}\Delta u.
\end{align}
Moreover, the mean curvature transforms under conformal transformation and change of unit normal via
\begin{align}\label{3.2}
\tilde H=-u^{\frac{-2}{n-2}}H-\frac {2(n-1)}{n-2}u^{\frac{-n}{n-2}}\nabla_\nu u.
\end{align}
Note that we have to change the sign of the unit normal of $\tilde \Sigma$ so that it corresponds with the normal of $\Sigma$ after gluing $\tilde M$ into $M$ as carried out below.
Identity \eqref{3.2} leads in combination with $u|_\Sigma=1$ to
\begin{align}\label{3.3}
\tilde H=-H-\frac {2(n-1)}{n-2}\nabla_\nu u.
\end{align}
Next we glue $M$ and $\tilde M$ together along $\Sigma$ to obtain a manifold $\hat M$ with corner $\Sigma$.
More precisely we define $(\hat M,\hat g)$ by $\hat M=(M\sqcup \tilde M)/_{\Sigma\sim \tilde \Sigma}$ and $\hat g|_M=g$, $\hat g|_{\tilde M}=\tilde g$. Then $\hat M$ is asymptotically flat, has non-negative scalar curvature and no boundary.
Furthermore, we have due to \eqref{eq-main-bdry-cond} and \eqref{3.3}
\begin{align*}
\tilde H=-H-\frac {2(n-1)}{n-2}\nabla_\nu u\ge H.
\end{align*}
Thus all the assumptions of the positive mass theorem with corners, Theorem \ref{thm-corner}, are satisfied and we obtain $\hat \m\ge0$.
Since we did not change the asymptotic behavior of the metric during the construction of $(\hat M,\hat g)$ we have $\m=\hat \m\ge0$ as desired.

Next, we apply the rigidity part of the positive mass theorem with corners \cite{corners}, \cite{corners2} to deduce that in the case of $\m=0$, $(\hat M, \hat g)$ is isometric to $(\R^n,\delta)$ where $\delta$ the standard metric.
Thus $(M^n,g)$ is isometric to $(R^n\setminus\Omega, \delta)$, and $(\tilde M^n,\tilde g)$ is isometric to $(\Omega,\delta)$ where $\Omega\subset \R^n$ is some smooth, open set.
Our goal is to show that $\Omega=B_\rho(x_0)$, $\rho>0$, where $x_0\in \tilde M$ is the compactified point from $\infty$.
We begin by observing that the function 
$$
\tilde u:=\frac1u
$$
is harmonic on $\tilde M$ with $\tilde u|_\Sigma=1$ and 
$$
\nabla_{\tilde \nu}\tilde u=-\nabla_\nu\tilde u=\nabla_\nu u.
$$
Again, note how the unit normal of $\Sigma$ changes sign after the conformal transformation.
Therefore, we can define $\hat u\in C^1(\hat M)$ by
$$
\begin{cases}
\hat u(x)=u(x)\quad\text{for $x\in M$},\\
\hat u(x)=\tilde u(x)\quad\text{for $x\in \tilde M$}.
\end{cases}
$$
Since $\hat u$ is harmonic outside $\Sigma$ and $C^1$ across $\Sigma$, it follows from standard PDE arguments that $\hat u$ is also harmonic across $\Sigma$. 
Thus, the uniqueness of solutions to the Laplace equation implies
$$
\hat u(x)=\frac {\rho^{n-2}}{|x-x_0|^{n-2}}
$$
for some $\rho>0$ which results in $\Sigma=S_\rho(x_0)$ as desired.
Alternatively, we could also apply directly Theorem 5.1 of \cite{capacity fill ins} to $(\tilde M,\tilde g,\tilde u)$.
\end{proof}

\section{The General Case}

In this section, we prove Theorem \ref{thm-intro} in the general case.
Let $u$ be the conformal Green's function from Theorem \ref{thm-intro}, i.e. $\mathcal L u=0$ where 
$$\mathcal L=\Delta-\frac{n-2}{4(n-1)}R$$ 
is the conformal Laplacian with respect to $g$.
As in Appendix A of \cite{capacity fill ins} we have the following asymptotic behavior of $u$ at $\infty$:

\begin{lemma}\label{lemma4.1}
There exists a constant $\mathcal D$ such that 
\[
u(x)=\frac{\mathcal D}{|x|^{n-2}}+\mathcal O(|x|^{-\gamma})
\]
where $\gamma:=\min(q-2,n+\tau-2,n-1)>n-2$.
\end{lemma}

\begin{proof}
Let $\R^n\setminus B_{R_1}(0)$ be the asymptotically flat end.
We start by construction a barrier to the conformal Green's function.
For this purpose let $\psi=ar^{-n+2}-r^{-n+2-\epsilon}$ with $a,\epsilon>0$.
 Then we have as in the proof of Lemma A2 in \cite{capacity fill ins}
\begin{align*}
\mathcal L\psi
=&(-n+2-\epsilon)(-n+1-\epsilon)r^{-n-\epsilon}+\mathcal O(|x|^{-\kappa}),
\end{align*}
where $\kappa:=\min(n+\tau,q)>n$.
Choosing $\epsilon<\min(\tau,n-q)$ there is an $R_2\ge R_1$ independent of $a$ such that $\Delta\psi\le0$ on $\R^3\setminus B_{R_2}(0)$ .
Moreover, we set $a\gg1$ such that $\psi>u$ on the complement of $\R^n\setminus B_{R_2}(0)$ in $M$.
Hence we have by the maximum principle $u\le\psi\le ar^{-n+2}$ and similarly $u\ge- ar^{-n+2}$.

By the Schauder Estimates, Theorem 6.2 in \cite{gilbargtrudinger}, we have
\begin{align*}
|u|\le  C|x|^{-\kappa},\quad |\nabla u|\le C|x|^{-\kappa-1},\quad|\nabla^2 u|\le C|x|^{-\kappa-2}
\end{align*}
 where $C$ denotes some constant which may change in the following computations from line to line.
We extend $u$ smoothly onto $\R^n$ such that $u\le C$ in $B_{R_0}(0)$ with $R_0\le R_1$.
Let $f=\hat \Delta u=\mathcal O(|x|^{-\kappa})$, where $\hat\Delta$ is the Euclidean Laplacian, and define
\begin{align*}
w(x):=-\frac1{n(n-2)\alpha_n}\int_{\R^n}\frac1{|x-y|^{n-2}}f(y)dy
\end{align*}
where $\alpha_n$ is the volume of the unit ball in $\R^n$. 
This is well defined due due to the decay properties of $f=\hat\Delta u$.
Also, observe that $\hat\Delta w=f$.
Next, we compute as in the proof of Lemma A2 in the Appendix A of \cite{capacity fill ins} denoting $|x|=r$

\begin{align}\label{4.1}
\int_{B_r(x)\setminus B_r(0)}\frac1{|x-y|^{n-2}}f(y)dy\le& Cr^{-\kappa}\int_{B_r(x)\setminus B_r(0)}\frac1{|x-y|^{n-2}}dy\le Cr^{-\kappa +2}
\end{align}
and
\begin{align}\label{4.2}
\int_{\R^n\setminus (B_r(0)\cup B_r(x))}\frac1{|x-y|^{n-2}}f(y)dy\le &Cr^{-n+2}\int_{\R^3\setminus( B_r(0)\cup B_r(x))}|y|^{-n-\tau}\le Cr^{-\kappa +2}.
\end{align}

Moreover, we additionally split the integral over $B_r(0)$ into $B_r(0)\setminus B_{R_0}(0)\cup B_{R_0}(0)$:

\begin{align}\label{4.3}
&\int_{ B_r(0)\setminus B_{R_0}(0)}\left(\frac1{|x-y|^{n-2}}-\frac1{|x|^{n-2}}\right)f(y)dy\nonumber\nonumber\\
=&\int_{B_{R_0}(0)\setminus B_\rho(0)}\frac{|x|^{2n-4}-|x-y|^{2n-4}}{|x|^{n-2}|x-y|^{n-2}(|x|^{n-2}+|x-y|^{n-2})}f(y)dy\nonumber\nonumber\\
\le&\sum_{j=1}^{2n-4}\int_{B_r(0)\setminus B_{R_0}(0)}C\frac{|y|^{-\kappa+j}}{|x|^{n-2+j}}
\le  C r^{-n+1}.
\end{align}

Similarly, we have
\begin{align}\label{4.4}
&\int_{ B_{R_0}(0)}\left(\frac1{|x-y|^{n-2}}-\frac1{|x|^{n-2}}\right)f(y)dy\nonumber\\
=&\int_{B_{R_0}(0)}\frac{|x|^{2n-4}-|x-y|^{2n-4}}{|x|^{n-2}|x-y|^{n-2}(|x|^{n-2}+|x-y|^{n-2})}f(y)dy\nonumber\\
\le&\sum_{j=1}^{2n-4}\int_{B_{R_0}(0)}C\frac{|y|^{k}}{|x|^{n-2+k}}
\le Cr^{-n+1}.
\end{align}

Lastly, we note
\begin{align}\label{4.5}
\frac1{|x|^{n-2}}\int_{\R^n\setminus B_r(0)}f(y)dy\le Cr^{-\kappa+2}.
\end{align}

Combining equations \eqref{4.1}-\eqref{4.5}, we obtain
\begin{align*}
-n(n-2)\alpha_nw=\frac{\int_{\R^n}\hat\Delta u}{r^{n-2}}+\mathcal O(|x|^{-\min(q-2,n+\tau-2,n-1)}).
\end{align*}

In particular, we have by the maximum principle $w=u$.
As in \cite{capacity fill ins}, Schauder theory, Theorem 6.3 in \cite{gilbargtrudinger}, gives the higher order estimates.
\end{proof}

Without loss of generality we may assume $R_1=1$, i.e. the asymptotically flat end is diffeomorphic to $\R^n\setminus B_1(0)$.
Next, we introduce the Kelvin transform $\mathcal K$ and prove the following elementary lemma:

\begin{lemma}\label{lemma-elementary}
Let $(M^n,g)$ be asymptotically flat, i.e. $g_{ij}=\delta_{ij}+\sigma_{ij}$ where $\sigma_{ij}=\mathcal O_2(|x|^{-\tau})=\mathcal O_2(|y|^{\tau}) $and $\mathcal K(x):=\frac x{|x|^2}$.
Then we have for  $y=\mathcal K(x)$ and $h_{ij}=g(\partial_iy,\partial_jy)$ 
\begin{align*}
h_{ij}=&|y|^{-4}\delta_{ij}+\mathcal O(|y|^{\tau-4}).
\end{align*}
Moreover, we have
\[
\partial_kh_{ij}=-4|y|^{-5}\partial_k|y|\delta_{ij}+\mathcal O(|y|^{\tau-5}).
\]
\end{lemma}

\begin{proof}
We compute
\begin{align*}
h_{ij}=&g\left(\frac{\partial_ix}{|x|^2}-2\frac{\langle\partial_i x,x\rangle x}{|x|^4},\frac{\partial_jx}{|x|^2}-2\frac{\langle\partial_j x,x\rangle x}{|x|^4}\right)
\\=&|y|^{-4}[\delta_{ij}+\sigma_{ij}+|y|^{-4}(4y_iy_jy_ky_l\sigma^{kl}-2|y|^2(y_iy^k\sigma_{jk}+y_jy^k\sigma_{ik}))]
\\=&|y|^{-4}\delta_{ij}+\mathcal O(|y|^{\tau-4}).
\end{align*}
The second statement follows analogously.
\end{proof}

Next, we study the metric $\tilde g=u^{\frac 4{n-2}}g$.
A priori $\tilde g $ is defined on $\R^n\setminus B_1(0)$ but by inverting the coordinates via the Kelvin transform $\mathcal K$ we may view $\tilde g$ as metric on $B_1(0)$.
Thereby the point $\infty$ of the one point compactification corresponds to the origin under the Kelvin transform and we wish to extend $\tilde g$ there.
This is done in the following lemma which is based on Lemma 6.1 in \cite{capacity fill ins}.
Moreover, we also obtain some regularity for $\tilde g$ which is necessary in order to perform the smoothing procedure below.

\begin{lemma} \label{lem-pt}
The metric $\tilde g$ extends continuously across the origin to a $W^{1,p}$ metric for some $p>n$.
\end{lemma}

\begin{proof}
We compute in the coordinates $y=\mathcal K(x)$ for $h_{ij}$
\begin{align*}
u^{\frac 4{n-2}}h_{ij}
=&\mathcal D\delta_{ij}+\mathcal O(|y|^{\gamma+2-n}).
\end{align*}
Hence $h$ is continuous in the origin.
Next, we compute in a similar fashion
\begin{align*}
\partial_k(u^{\frac 4{n-2}}h_{ij})=&(4\mathcal D|y|^{n-7}+\mathcal O(|y|^{\gamma-5}))(\mathcal D|y|^{n-2}+\mathcal O(|y|^{\gamma}))^{\frac {6-n}{n-2}}\partial_k|y|
\\&+(-4\mathcal D|y|^{n-7}+\mathcal O(|y|^{\gamma-5})(\mathcal D|y|^{n-2}+\mathcal O(|y|^{\gamma}))^{\frac {6-n}{n-2}}\partial_k|y|
\\=&\mathcal O(|y|^{\gamma-n+1}).
\end{align*}
Thus $u^4h_{ij}\in W^{1,p}(B_1(0))$ for $p= \frac{n}{\min(q-n-1,\tau-1)}>n$.
\end{proof}

Now we are in a position to verify the following Proposition, which is known to 
people who are familiar with the work in  \cite{corners, shi tam2, lee, corners2, shi tam, singularPMT, capacity fill ins}.

\begin{proposition} \label{prop-corner-pt}
Suppose $(M^n, g)$ has corner singularity along a  hypersurface 
$\Sigma$ with $ H_1 \le H_2 $  as in Theorem \ref{thm-corner}. 
Also,  suppose  there is  a point singularity $ q \in  M_2$ where $g$ is  in $W^{1,p}$ for some $p>n$ near $q$.
If $g$ has  non-negative scalar curvature away from $\Sigma$ and $ \{ q \}$,
then $\m\ge0$ with equality if and only if $(M_1,g)$ is isometric to $(\R^n\setminus \Omega,\delta)$ 
for a bounded domain $\Omega$ with smooth boundary.
\end{proposition}

\begin{proof}
Due to the assumption $H_1\le H_2$ at  $\Sigma$, we may exactly follow  Proposition 3.1  in \cite{corners} to approximate $g$ by a family of smooth metrics $g_\delta$ such that $g_\delta(x)=g(x)$ for $\operatorname{dist}(\Sigma,x)\ge \delta$ and the scalar curvature $R_\delta$ satisfies
\begin{align}\label{4.6}
R_\delta(x)\ge -C
\end{align}
for $\operatorname{dist}(\Sigma,x)< \delta$. 
Equation \eqref{4.6} in particular shows that the integral of the  negative part of the scalar curvature  can be made arbitrarily small during the approximation process.

Near the point singularity $q$, we approximate $g$ as in \cite{shi tam} Lemma 4.1 and \cite{capacity fill ins} Lemma 3.6 to obtain smooth metrics $ \{g_\epsilon\}$ so that 
 $g_\epsilon=g$ outside $B_\epsilon(\{q\})$, $\|g_\epsilon\|_{W^{1,p}(B_\epsilon(\{q\})}\le C$ where $C$ is independent of $\epsilon$.
By  Lemma 3.7 in \cite{capacity fill ins},   the uniform $W^{1,p}$ bound on $g_\epsilon$ implies that the integral of the negative part of $ R_{g_\epsilon} $ over $B_\epsilon(\{q\})$ becomes arbitrarily small.
(We  note that Lemma 3.7 in \cite{capacity fill ins} is stated in a slightly more general version and for our purpose the result already follows from the estimate on the term  $\bold I_B$ in Lemma 3.7 in \cite{capacity fill ins}.)

Since both approximations $g_\epsilon$ and $g_\delta$ are local, we can perform them simultaneously to approximate $g$ by a smooth metric with $\int_M R^-$ becoming arbitrarily small. Here $R^-$ denotes the negative 
part of the scalar curvature. 

Now we proceed as usual and conformally transform $M$ such that $R\ge0$ everywhere, see \cite{old7} and section 4  in \cite{corners}.
Thereby the mass converges as in \cite{corners} and we may deduce the mass is non-negative by the standard positive mass theorem.

If the mass is zero, we need to apply the argument  from \cite{corners2}.
This is due to \cite{corners2} relying on Ricci flow which has the advantage that the mass stays constant during the smoothing process  so we can apply the rigidity statement of the  positive mass theorem to the 
flow solution with initial data $g$
(see also section 7 of \cite{shi tam}).
This shows the rigidity of the Proposition. 
\end{proof}

\begin{proof}[Proof of Theorem \ref{thm-intro}]
Due to Lemma \ref{lem-pt}, the conformally filled in manifold $(\hat M, \hat g) $  constructed in 
Section \ref{sec-H-flat} satisfies the conditions of Proposition \ref{prop-corner-pt}.
Hence Theorem \ref{thm-intro} follows from Proposition \ref{prop-corner-pt} in the same way that the harmonically flat case 
is proven in Section \ref{sec-H-flat}.
\end{proof}

\section{Application} \label{sec-app}

In this section, we prove Theorem \ref{thm-sm} which generalizes H.~Bray's result in Theorem \ref{thm-bray}.

\begin{proof}[Proof of Theorem \ref{thm-sm}]

By the maximum principle, $ \phi > -1 $ on $M$. 
Define 
$$   w =  \frac{ 2}{1+ \phi } \ \ \mathrm{and} \ \ 
 \tilde g = w^{- \frac{4}{n-2} } g . $$
$ (M^n, \tilde g)$ is asymptotically flat with nonnegative scalar curvature. 
The fact $ \Delta w^{-1} = 0 $ implies 
$$
\tilde \Delta w = 0 . 
$$
Moreover, $ w \to 1$ at $\infty$ and $ w = \frac{2}{1+c} $ at $ \Sigma$.
Next, define 
$$
v = \frac{1+c}{1-c} ( w - 1) . 
$$
Then 
$$
\tilde \Delta v = 0, \ \ 
v \to 0 \ \mathrm{at} \ \infty, \ \ 
\mathrm{and} \ v = 1 \ \mathrm{at} \ \Sigma . 
$$
Let $ \tilde H$ be the mean curvature of $ \Sigma$ 
in $(M, \tilde g)$ with respect to the $\infty$-pointing unit normal $\tilde \nu$.
By (3.2), it follows that
\begin{align*}
\tilde H =& w^\frac{2}{n-2} \left[ H + \frac{2(n-1)}{(n-2)} w\nabla_\nu w^{-1} \right]  ,\\
 \tilde \nu =& w^\frac{2}{n-2} \nu  ,
\end{align*}
and
$$
\nabla_{\tilde \nu} v= w^\frac{2}{n-2} \nabla_\nu v 
$$
Hence, at $ \Sigma$, by \eqref{eq-bd-c-f}, 
\begin{align*}
\begin{split}
& \ \tilde H + \frac{n-1}{n-2} \nabla_\nu v \\
= & \ w^\frac{2}{n-2} \left[ H + \frac{(n-1)}{(n-2)} 
 \left( 2 w \nabla_\nu w^{-1} + \nabla_\nu v \right) \right] \\
= & \ w^\frac{2}{n-2} \left[ H -  \frac{(n-1)}{(n-2)} 
 \frac{2c }{ (1 - c^2 )} \nabla_\nu v\right] \\
 \le & \ 0 .
\end{split}
\end{align*}
Thus, by Corollary \ref{cor-intro}, 
\begin{align}\label{5.1} 
 \tilde \m \ge 0
\end{align}
where $ \tilde \m $ is the ADM mass of $(M^n, \tilde g)$.
Since by the definition of mass  $ \tilde \m $ and $ \m $ are related by $\tilde \m = \m - C $, we conclude from equation \eqref{5.1}
that
$$
\m \ge C .
$$

If $ \m = C $, then $ \tilde \m = 0 $. By Corollary \ref{cor-intro}, $(M^n, \tilde g)$ is isometric to
$ \left(  \R^n \setminus \{ | x | < r_0 \} , \delta \right) $ for some constant $ r_0 > 0 $.
In this case, $ w $ is a Euclidean  harmonic function that goes to $ 1$ at $\infty$ and 
equals a positive constant $ \frac{2}{1+ c} $ at $ \Sigma = \{ | x | = r_0 \}$. Hence, 
$$
w = 1 + \frac{ \m}{2  | x|^{n-2} }  ,
$$
where $ \m$ is a constant satisfying $ \displaystyle \m =   \frac{ (1-c) }{1+c}  2 r_0^{n-2} $.
This completes the proof.
\end{proof}

It is worth of pointing out the following equivalent form of Theorem \ref{thm-sm}.

\begin{theorem}  \label{thm-sm-2}
Let $(M^n,g)$ be an $n$-dimensional, asymptotically flat  manifold, with non-negative scalar curvature, with  boundary $\Sigma$. 
Let $\varphi (x) $ be a function on $(M, g)$ which satisfies
\begin{align*}
\left\{ 
\begin{array}{rcc}
\Delta \varphi & = & 0  \ \ \mathrm{in} \ M \\
\varphi & \to & 1 \ \ \mathrm{at} \ \infty \\ 
\varphi & = & 0 \ \ \mathrm{at} \ \Sigma .
\end{array}
\right.
\end{align*}
Let $ H$ be the mean curvature of $\Sigma$ in $(M^n, g)$ with respect
to the $\infty$-pointing normal $\nu$. 
If there is a constant $ c > -1$ such that 
\begin{align}\label{5.2}
\frac{2c}{1 + c } \nabla_\nu\varphi \ge \frac{n-2}{n-1} H,
\end{align}
then 
\begin{align}\label{5.3}
 \m  \ge (1-c) \mathcal{C} ,
\end{align}
where $\m $ is the ADM mass of $(M^n, g)$ and  $ \mathcal{C}$ is the constant in
$$
\varphi = 1 - \frac{\mathcal{C}}{|x|^{n-2} } + o(|x|^{2-n} ), \ \mathrm{as} \ x \to \infty. 
$$
Moreover, equality in \eqref{5.3} holds if and only if $(M^n, g)$ is isometric to 
an $n$-dimensional spatial Schwarzschild manifold outside a rotationally symmetric sphere, i.e. 
$$ \left( \R^n \setminus \left\{ | x | < r_0 \right\} , \left( 1 + \frac{\m}{2 |x|^{n-2} } \right)^\frac{4}{n-2} \delta_{ij}  \right)\quad\text{for some constants $ r_0 > 0 $.}
$$
\end{theorem}

\begin{proof}
If $ c \neq 1$, the theorem follows  directly from Theorem \ref{thm-sm} by letting $ \varphi = \frac{1}{1 - c} (\phi - c) $. If $c = 1$, the theorem reduces to Corollary \ref{cor-intro}.
\end{proof}

\begin{remark}
For a constant $c > -1 $ satisfying \eqref{5.2} to exist, one needs to have
\[
2 \nabla_\nu\varphi  >  \frac{n-2}{n-1} H.
\]
If $n=3$,  this coincides with  the condition used in \cite[Theorem 1.5]{capacity fill ins}.
\end{remark}

\section{Examples and Discussions} \label{sec-end}

\begin{example}[Boundary with $H\le0$]
Every manifold $(M^n, g)$ whose boundary  $\p M$ has non-positive mean curvature satisfies condition \eqref{eq-main-bdry-cond} and thus has  positive mass $m$. 
More precisely, $ m \ge \mathcal C > 0$ by Theorem \ref{thm-sm}.
\end{example}

\begin{example}[Regions in  Schwarzschild manifold]
Let $(M^n, g) = ( \R^n\setminus B_{r_0}(0), (1+\frac {m}{2r^{n-2}})^{\frac4{n-2}}\delta$)
be a spacelike slice in Schwarzschild spacetime of mass $m$.
Note, $m$ is allowed to be negative.
We begin by computing the mean curvature of the boundary $S_{r_0}(0)$ using \eqref{3.2}
\begin{align*}
 H=&(2^{\frac{n}{n-2}}r_0^{n-1}-2^{\frac2{n-2}}mr_0)\frac{n-1}{(2r_0^{n-2}+m)^{\frac n{n-2}}}.
\end{align*}

Next, we observe that the Green's function is given by
\begin{align}\label{6.1}
u(r)=\frac{2r_0^{n-2}+m}{2r^{n-2}+m}.
\end{align}
Combining this with the observation $|\partial_r|_g=(1+\frac m{2r^{n-2}})^{\frac2{n-2}}$, we have at $S_{r_0}$
\begin{align*}
-\frac{n-1}{n-2}\nabla_\nu u
=&2^{\frac n{n-2}}r_0^{n-1}\frac{n-1}{(2r_0^{n-2}+m)^{\frac n{n-2}}}.
\end{align*}
Thus we have as predicted by the rigidity statement of Theorem \ref{thm-sm}
\begin{align*}
H=-\frac{2c(n-1)}{(1+c)(n-2)}\nabla_\nu u
\end{align*}
for 
\[
c=\frac{2r_0^{n-2}}{2r_0^{n-2}+m}.
\]
Moreover, $\frac{2c}{c+1}$ can approach $1$ while maintaining negative mass which shows that Theorem \ref{thm-intro} is sharp in this sense.
\end{example}

\begin{example}[Conformal minimal boundary] \label{ex-conf-m-bdry}
Suppose that $g=\phi^{\frac{4}{n-2}}\tilde g$ such that $\Sigma $ is a minimal surface with respect to $\tilde g$ where $\phi$ solves
$$
\left\{ 
\begin{array}{rcc}
\Delta \phi & = & 0  \ \ \mathrm{in} \ M \\
\phi & \to & 2 \ \ \mathrm{at} \ \infty \\ 
\phi & = & 1 \ \ \mathrm{at} \ \Sigma .
\end{array}
\right.
$$

Observe that $\tilde R$ has by equation \eqref{3.1} also non-negative scalar curvature.
Note that $v\phi$ is a harmonic function on $(M^n,\tilde g)$ satisfying $(\phi v)|_\Sigma=1$ and $(\phi v)(x)\to0$ for $|x|\to\infty$.
Precisely the same properties hold true for the function $(2-\phi)$ and thus we deduce by the uniqueness property for harmonic functions $\phi v=2-\phi$.
Therefore, we have at $\Sigma$
\[
\nabla_\nu v =-\frac2{v^2}\nabla_\nu \phi=-2\nabla_\nu \phi.
\]
Next, we compute using equation \eqref{3.2}
\[
H=\tilde H \phi^{-\frac 2{n-2}}+\frac{2(n-1)}{n-2}\phi^{-\frac n{n-2}}\nabla_\nu \phi=\frac{2(n-1)}{n-2}\nabla_\nu \phi=-\frac{n-1}{n-2}\nabla_\nu v.
\]
Hence condition (1.2) is satisfied with equality and we have $m\ge0$.
Conversely, suppose we have $H\equiv -\frac{n-1}{n-2}\nabla_\nu v$. 
Then, we define $\phi:=\frac2{1+v}$ and $\tilde g=\phi^{\frac{-4}{n-2}}g$.
Applying again \eqref{3.2}, we obtain
\begin{align*}
\tilde H=&H \phi^{\frac2{n-2}}+\frac {2(n-1)}{n-2}\phi^{\frac n{n+2}}\nabla_\nu (\phi^{-1})=H+\frac{n-1}{n-2}\nabla_\nu v=0.
\end{align*}
Thus $\Sigma$ is a minimal surface with respect to $\tilde g$.
Indeed, this calculation shows that  condition \ref{condition2}  holds if and only if  $\tilde H\le0$ under the above conformal transformation.
In particular, this suggests Bray's result and proof of Theorem \ref{thm-bray} in \cite{penrose2} can also be applied to 
derive Corollary \ref{cor-intro} and Theorem \ref{thm-sm} in the harmonically flat case.
\end{example}

Next, we want to relate Theorem \ref{thm-intro}, Theorem \ref{thm-sm} and Example \ref{ex-conf-m-bdry} to the Riemannian Penrose conjecture.
This conjecture was  shown by G.~Huisken and T.~Ilmanen \cite{penrose} (one black hole) and by H.~Bray \cite{penrose2} (multiple black holes).
The latter proof has been extended up to dimension 7 by H.~Bray and D.~Lee \cite{penrose3}.

One of the key steps in H.~Bray's proof is to prove the mass-capacity estimate, Theorem \ref{thm-bray}.
There an asymptotically flat manifold with non-negative scalar curvature is first reflected across its horizon, then conformally transformed via the Green's function so the positive mass theorem may be applied.
In view of the positive mass theorem for manifolds with boundary, the reflection argument can now be dropped.
More precisely, the asymptotically flat manifold can be directly conformally transformed via the Green's function as presented in Example \ref{ex-conf-m-bdry}.

As it turns out Theorem \ref{thm-bray} is the reason the proof of the Penrose inequality breaks down in higher dimension as discovered by H.~Bray.
This is due to minimal hypersurfaces being no longer regular which then in particular complicates the reflection argument.
Therefore, showing a positive mass theorem for manifolds with singular boundary may be considered a strategy for proving the higher dimensional Penrose inequality.

We also recall that M.~Herzlich \cite{penrose like, Herzlich3} showed a  positive mass theorem 
for asymptotically flat manifolds with boundary $\partial M$ whose mean curvature satisfies 
an upper bound which depends on the area of $ \partial M$ in $3$-dimension and 
depends on the Yamabe invariant $\mathcal Y$ of $\partial M$ in higher dimensions. 
More precisely, the latter case assumes
$$
H\le |\Sigma|^{-\frac1{n-1}}\sqrt{\frac{n-1}{n-2}\mathcal Y(\Sigma)}.
$$
The proofs in  \cite{penrose like, Herzlich3} make use of  estimates of the
first eigenvalue of the Dirac operator by  C.~Bär \cite{diraceigenvalue1} 
and by O. Hijazi \cite{diraceigenvalue2, HO2}. 
Results in  \cite{penrose like, Herzlich3} have been extended to the 
asymptotically hyperbolic setting by O. Hijazi, S. Montiel and S. Raulot \cite{HMR}.
It is plausible that Theorem \ref{thm-intro} or Corollary \ref{cor-intro} may have an analogue in 
the asymptotically hyperbolic setting.


\begin{thebibliography}{40}

\bibitem{ADM}
Arnowitt, R.; Deser, S.; Misner, C.:
\textit{Dynamical structure and definition of
energy in general relativity}.
Phys. Rev., vol.116, 1322–1330, 1959.

\bibitem{diraceigenvalue1}
Bär, C.:
\textit{Lower eigenvalues estimates for Dirac operators}.
Math. Ann., vol. 293, 39-46, 1992.



\bibitem{witten3}
Bartnik, R.:
\textit{The Mass of an Asymptotically Flat Manifold}.
Commun. Pure and Applied Math., vol. 39, 661-693, 1986.


\bibitem{penrose2}
Bray, H.:
\textit{Proof of the Riemannian Penrose Inequality Using the Positive Mass Theorem}.
J. Differential Geometry, vol. 59, 177-267, 2001.

\bibitem{bray-jauregui}
Bray, H.; Jauregui, J.:
\textit{A geometric theory of zero area singularities in general relativity}.
Asian J. Math., vol. 17, no. 3, 525-560, 2013.


\bibitem{penrose3}
Bray, H.; Lee, D.:
\textit{On the Riemannian Penrose Inequality in Dimensions less than Eight}.
Duke Math. J., vol. 148, 81-106, 2009.


\bibitem{blackhole}
Bunting, G.; Masood-ul-Alam, A.:
\textit{Non-Existence of Multiple Black Holes in Asymptotically Euclidean Static Vacuum Space-Time}.
 Gen. Rel. and Grav. vol. 19, 147-154, 1987.

\bibitem{Chrusciel}
Chru\'{s}ciel, P.T.:
\textit{Boundary conditions at spatial infinity from a Hamiltonian point of view}. 
Topological properties and global structure of space-time (Erice, 1985), NATO Adv. Sci. Inst.
Ser. B Phys., vol. 138,  49-59, Plenum, New York, 1986.

\bibitem{Escobar}
Escobar, J.:
\textit{The Yamabe Problem on Manifolds with Boundary}
J. Differential Geometry, vol. 35, 21-84, 1992.

\bibitem{blackhole pmt}
Gibbons, G.; Hawking, S.; Horowitz, G.; Perry, M.:
\textit{Positive Mass Theorems for Black Holes}.
Comm. Math. Phys., vol. 88, 295-308, 1983

\bibitem{gilbargtrudinger}
Gilbarg, D.; Trudinger, N.:
\textit{Elliptic Partial Differential Equations of Second Order}
Classics in Mathematics, Springer-Verlag, 2001.

\bibitem{penrose like}
Herzlich, M.:
\textit{A Penrose-like inequality for the mass of Riemannian asymptotically flat manifolds}.
Commun. Math. Phys., vol. 188, 121-133, 1997.

\bibitem{blackhole pmt2}
Herzlich, M.:
\textit{The positive mass theorem for black holes revisited}.
J. of Geometry and Physics, vol. 26, 97-111, 1998

\bibitem{Herzlich3}
Herzlich, M.: 
\textit{Minimal spheres, the Dirac operator and the Penrose inequality}. 
S\'{e}minaire de Th\'{e}orie Spectrale et G\'{e}om\'{e}trie (Institut Fourier, Grenoble), vol. 20, 
9-16, 2002. 



\bibitem{diraceigenvalue2}
Hijazi, O.:
\textit{A conformal lower bound for the smallest eigenvalue of the Dirac operator and killing spinors}.
Commun. Math. Phys., vol. 104, 151-162, 1986.

\bibitem{HO2}
Hijazi, O.:
\textit{
Premi\`{e}re valeur propre de l'op\'{e}rateur de Dirac et nombre de Yamabe}. 
C. R. Acad. Sci. Paris, vol. 313,  865-868, 1991.

\bibitem{HMR}
Hijazi, O.; Montiel, S.; Raulot, S:
\textit{A positive mass theorem for asymptotically hyperbolic manifolds with inner boundary}.
Int. J. Math, vol. 26, no. 12, 1550101 (17 pages), 2015.

\bibitem{penrose}
Huisken, G.; Ilmanen, T.:
\textit{The Inverse Mean Curvature Flow and the Riemannian Penrose Inequality}.
J. Differential Geometry, vol. 59, 353-437, 2001.



\bibitem{lee}
Lee, D.A.:
\textit{A positive mass theorem for Lipschitz metrics with small singular sets}.
Proc. Amer. Math. Soc., vol. 141, 3997-4004, 2013.

\bibitem{lee-lefloch}
Lee, D.A.; LeFloch, P.G.:
\textit{The positive mass theorem for manifolds with distributional curvature}.
Commun. Math. Phys., vol. 339, 99-120, 2015.

\bibitem{Lo3}
Lohkamp,  J.:
\textit{The Higher Dimensional Positive Mass Theorem I}.
arXiv:math/0608795v2.

\bibitem{Lo4}
Lohkamp,   J.:
\textit{Skin    Structures    in    Scalar    Curvature    Geometry}.
arXiv:1512.08252.


\bibitem{singularPMT}
Mantoulidis, C.; Li, C.:
\textit{Positive Scalar Curvature with Skeleton Singularities}. 
Mathematische Annalen, vol. 374, 99-131, 2019.


\bibitem{capacity fill ins}
Mantoulidis, C.; Miao, P.; Tam, L.-F.:
\textit{Capacity, Quasi-local Mass, and Singular Fill-ins}.
Journal f\"{u}r die reine und angewandte Mathematik (Crelles Journal), 
DOI: https://doi.org/10.1515/crelle-2019-0040.

\bibitem{corners2}
McFeron, D.; Székelyhidi, G.:
\textit{On the positive Mass Theorem for Manifolds with Corners}.
Commun. Math. Phys., vol. 313, 425-443, 2012.

\bibitem{corners}
Miao, P.:
\textit{Positive Mass Theorem on Manifolds admitting Corners along a Hypersurface}.
Adv. Theor. Math. Phys., vol. 6, 1163-1182, 2002.


\bibitem{witten2}
Parker, T.; Taubes, H.:
\textit{On Witten's Proof of the Positive Energy Theorem}.
Commun. Math. Phys., vol. 84, 223-238, 1982



\bibitem{old}
Schoen, R.; Yau, S.-T.:
\textit{On the proof of the Positive Mass Conjecture in General Relativity}.
Commun. Math. Phys., vol. 65, 45-76, 1979.

\bibitem{old7}
Schoen, R.; Yau, S.-T.:
\textit{Complete manifolds with nonnegative scalar curvature and the positive action conjecture in general relativity}.
Proc. Nat. Acad. Sci. U.S.A., vol. 76, 1024-1025, 1979.


\bibitem{PMT}
Schoen, R.; Yau, S.-T.:
\textit{Positive Scalar Curvature and Minimal Hypersurface Singularities}.
arXiv:1704.05490.


\bibitem{shi tam2}
Shi, Y.; Tam L.-F.:
\textit{Positive Mass Theorem and the Boundary Behaviors of Compact Manifolds with Nonnegative Scalar Curvature}.
J. Differential Geom., vol. 62, 79-125, 2002.

\bibitem{shi tam}
Shi, Y.; Tam L.-F.:
\textit{Scalar Curvature and Singular Metrics}.
Pacific J. Math., vol. 293, 427-470, 2018.

\bibitem{witten} 
Witten, E.:
\textit{A new Proof of the Positive Energy Theorem}. 
Commun. Math. Phys., 80, vol. 3, 381-402, 1981.

\end{thebibliography}
\end{document}